\documentclass{article}
\usepackage{amsfonts}
\usepackage{amsmath}
\usepackage{amssymb}
\usepackage{graphicx}
\usepackage{amsopn}
\usepackage{amsthm}
\usepackage{url} 
\usepackage{color}
\usepackage{hyperref} 
\usepackage[all]{xy}
\usepackage{amsrefs}
\usepackage{etoolbox, xspace}

\newtheorem{theorem}{Theorem}[section]

\newtheorem{claim}[theorem]{Claim}

\newtheorem{conjecture}[theorem]{Conjecture}
\newtheorem{corollary}[theorem]{Corollary}
\newtheorem{lemma}[theorem]{Lemma}

\newtheorem{proposition}[theorem]{Proposition}
\theoremstyle{definition}
\newtheorem{definition}[theorem]{Definition}
\newtheorem{example}[theorem]{Example}

\newtheorem{question}[theorem]{Question}
\newtheorem{remark}[theorem]{Remark}
\input{tcilatex}
\newenvironment{thm}{\begin{theorem}}{\end{theorem}}
\newenvironment{cor}{\begin{corollary}}{\end{corollary}}

\newenvironment{pro}{\begin{proposition}}{\end{proposition}}
\newenvironment{lem}{\begin{lemma}}{\end{lemma}}
\newenvironment{df}{\begin{definition}}{\end{definition}}
\newenvironment{exa}{\begin{example}}{\end{example}}
\newenvironment{rem}{\begin{remark}}{\end{remark}}
\newcommand{\la}{\langle}
\newcommand{\ra}{\rangle}

\renewcommand{\iff}{\leftrightarrow}

\newcommand{\mc}{\mathcal}

\newcommand{\mf}{\mathfrak}

\renewcommand{\to}{\rightarrow}

\newcommand{\nil}{\varnothing}

\newcommand{\Le}{\preceq}

\DeclareMathOperator{\RT}{RT}
\DeclareMathOperator{\WKL}{WKL}
\DeclareMathOperator{\GS}{GS}
\DeclareMathOperator{\ACA}{ACA}
\DeclareMathOperator{\GSD}{GSD}
\DeclareMathOperator{\RCA}{RCA}
\DeclareMathOperator{\PCA}{\Pi^1_1-CA}
\DeclareMathOperator{\Eq}{Eq}
\newtoggle{draftmode}
\toggletrue{draftmode}

\author{
	Katie Brodhead \and 
	Mushfeq Khan \and
	Bj\o rn Kjos-Hanssen \and
	William A.\ Lampe \and
	Paul Kim Long V.\ Nguyen \and
	Richard A.\ Shore
}
\title{The Strength of the Gr\"atzer-Schmidt theorem}

\begin{document}
	\maketitle

	\begin{abstract}
		The Gr\"atzer-Schmidt theorem of lattice theory states that each algebraic lattice is isomorphic to the congruence lattice of an algebra.
		We study the reverse mathematics of this theorem. We also show that 
		\begin{enumerate}
			\item the set of indices of computable lattices that are complete is $\Pi^1_1$-complete; 
			\item the set of indices of computable lattices that are algebraic is $\Pi^1_1$-complete;
			\item the set of compact elements of a computable lattice is $\Pi^{1}_{1}$ and can be $\Pi^1_1$-complete; and
			\item the set of compact elements of a distributive computable lattice is $\Pi^{0}_{3}$, and there is an algebraic distributive computable lattice such that the set of its compact elements is $\Pi^0_3$-complete.
		\end{enumerate}
	\end{abstract}
	\tableofcontents

	\section{Introduction}

		The Gr\"atzer-Schmidt theorem \cite{MR0151406}, also known as the \emph{congruence lattice representation theorem}, states that each algebraic lattice is isomorphic to the congruence lattice of an algebra.
		It established a strong link between lattice theory and universal algebra. In this article we analyze the theorem from the point of view of reverse mathematics and calibrate the strength of the special case of the theorem for distributive lattices. The question of the strength of the general case of the theorem remains open. 

		We use notation associated with partial computable functions, $\varphi_e$, $\varphi_{e,s}$, $\varphi_{e,s}^\sigma$, $\varphi_e^f$ as in Odifreddi \cite{MR982269}.
		A $\Pi^1_1$ subset of $\omega$ may be written in the form (see, for example, Sacks \cite{MR1080970}, page 5)
		\[
			C_e=\{n\in\omega \mid \forall f\in\omega^\omega\,\,\varphi_e^f(n)\downarrow\}.
		\]
		A subset $A\subseteq\omega$ is \emph{$\Pi^1_1$-hard} if each $\Pi^1_1$ set is $m$-reducible to $A$; that is,
		for each $e$, there is a computable function $f$ such that for all $n$, $n\in C_e$ iff $f(n)\in A$.
		$A$ is \emph{$\Pi^1_1$-complete} if it is both $\Pi^1_1$ and $\Pi^1_1$-hard. It is well known that such sets exist. Fix for the rest of the paper a number $e_0$ so that $C_{e_0}$ is $\Pi^1_1$-complete.
		With each $n$, the set $C_{e_0}$ associates a tree $T'_n$ defined by
		\[
			T'_n=\{\sigma\in\omega^{<\omega}\mid \varphi_{e_0,|\sigma|}^\sigma(n)\uparrow\}.
		\]
		Note that $T'_n$ has no infinite path iff $n\in C_e$. 

		A \emph{computable lattice} $(L,\Le)$ has underlying set $L=\omega$ and a computable lattice ordering $\Le$ that is formally a subset of $\omega^2$. 

		We will use the symbol $\Le$ for lattice orderings, and reserve the symbol $\le$ for the natural ordering of the ordinals and in particular of $\omega$.
		Meets and joins corresponding to the order $\preceq$ are denoted by $\wedge$ and $\vee$. Below we will seek to build computable lattices from the trees $T'_n$.
		Since for many $n$, $T'_n$ will be finite, and a computable lattice must be infinite according to our definition, we will work with the following modification of $T'_n$:
		\[
			T_n=T'_n\cup \{\la i\ra: i\in\omega\}\cup \{\nil\}
		\]
		where $\nil$ denotes the empty string and $\la i\ra$ is the string of length 1 whose only entry is $i$. This ensures that $T_n$ has the same infinite paths as $T'_n$, and each $T_n$ is infinite.
		Moreover the sequence $\{T_n\}_{n\in\omega}$ is still uniformly computable.

	\section{Computability-theoretic analysis of lattice theoretic concepts}
		\subsection{Index set of complete lattices is \texorpdfstring{$\Pi^{1}_{1}$}{Pi11}-complete}
			\begin{df}
				A lattice $(L,\Le)$ is \emph{complete} if for each subset $S\subseteq L$, both $\sup S$ and $\inf S$ exist.
			\end{df}
			\begin{exa}
				In set-theoretic notation, $(\omega+1,\le)$ is complete. Its sublattice $(\omega,\le)$ is not, since $\omega=\sup\omega\not\in\omega$.
			\end{exa}
			\begin{lem}\label{1}
				The set of indices of computable lattices that are complete is $\Pi^1_1$.
			\end{lem}
			\begin{proof}
				The statement that $\sup S$ exists is equivalent to a first order statement in the language of arithmetic with set variable $S$:
				\[
					\exists a [\forall b (b\in S\to b\Le a) \And \forall c ((\forall b(b\in S\to b\Le c)\to a\Le c)].
				\]
				The statement that $\inf S$ exists is similar, in fact dual. Thus the statement that $L$ is complete consists of a universal set quantifier over $S$, followed by an arithmetical matrix.
			\end{proof}

			\begin{pro}\label{completeHard}
				The set of indices of computable lattices that are complete is $\Pi^1_1$-hard.
			\end{pro}
			\begin{proof}
				Let $L_n$ consist of two disjoint copies of $T_n$, called $T_n$ and $T_n^*$. For each $\sigma\in T_n$, its copy in $T_n^*$ is called $\sigma^*$. Order $L_n$ so that $T_n$ has the prefix ordering 
				\[
					\sigma\Le \sigma^\frown\tau,
				\]
				$T_n^*$ has the reverse prefix ordering, and $\sigma\prec\sigma^*$ for each $\sigma\in T_n$. We take the transitive closure of these axioms to obtain the order of $L_n$; see Figure \ref{fig:Ln}. 

				Next, we verify that $L_n$ is a lattice. For any $\sigma$, $\tau\in T_n$ we must show the existence of
				(1) $\sigma\vee\tau$, (2) $\sigma\wedge\tau$, (3) $\sigma\vee\tau^*$, and (4) $\sigma\wedge\tau^*$; the existence of $\sigma^*\vee\tau^*$ and $\sigma^*\wedge\tau^*$ then follows by duality. 

				We claim that for any strings $\alpha$, $\sigma\in T_n$, we have $\alpha^*\succeq\sigma$ iff $\alpha$ is comparable with $\sigma$; see Figure \ref{fig:Ln}. 
				In one direction, if $\alpha\succeq\sigma$ then $\alpha^*\succeq\alpha\succeq\sigma$, and if $\sigma\succeq\alpha$ then $\alpha^*\succeq\sigma^*\succeq\sigma$.
				In the other direction, if $\alpha^*\succeq\sigma$ then by the definition of $\Le$ as a transitive closure there must exist $\rho$ with $\alpha^*\succeq\rho^*\succeq\rho\succeq\sigma$.
				Then $\alpha\Le\rho$ and $\sigma\Le\rho$, which implies that $\alpha$ and $\rho$ are comparable.

				Using the claim we get that (1) $\sigma\vee\tau$ is $(\sigma\wedge\tau)^*$, where (2) $\sigma\wedge\tau$ is simply the maximal common prefix of $\sigma$ and $\tau$;
				(3) $\sigma\vee\tau^*$ is $\sigma^*\vee\tau^*$ which is $(\sigma\wedge\tau)^*$; and (4) $\sigma\wedge\tau^*$ is $\sigma\wedge\tau$. 

				It remains to show that $(L_n,\preceq)$ is complete iff $T_n$ has no infinite path. So suppose $T_n$ has an infinite path $S$.
				Then $\sup S$ does not exist, because
					$S$ has no greatest element,
					$S^*$ has no least element,
					each element of $S^*$ is an upper bound of $S$, and
					there is no element above all of $S$ and below all of $S^*$. 

				Conversely, suppose $T_n$ has no infinite path and let $S\subseteq L_n$. If $S$ is finite then $\sup S$ exists.
				If $S$ is infinite then since $T_n$ has no infinite path, there is no infinite linearly ordered subset of $L_n$, and so $S$ contains two incomparable elements $\sigma$ and $\tau$.
				Because $T_n$ is a tree, $\sigma\vee\tau$ is in $T_n^*$.
				Now the set of all elements of $L_n$ that are above $\sigma\vee\tau$ is finite and linearly ordered, and contains all upper bounds of $S$.
				Thus $S$ has a supremum. Since $L_n$ is self-dual, i.e. $(L_n,\Le)$ is isomorphic to $(L_n,\succeq)$ via $\sigma\mapsto\sigma^*$, infs also always exist. So $L_n$ is complete.
			\end{proof}

			\begin{figure}[htb!]
				\centering%
				\includegraphics[height=7.5cm]{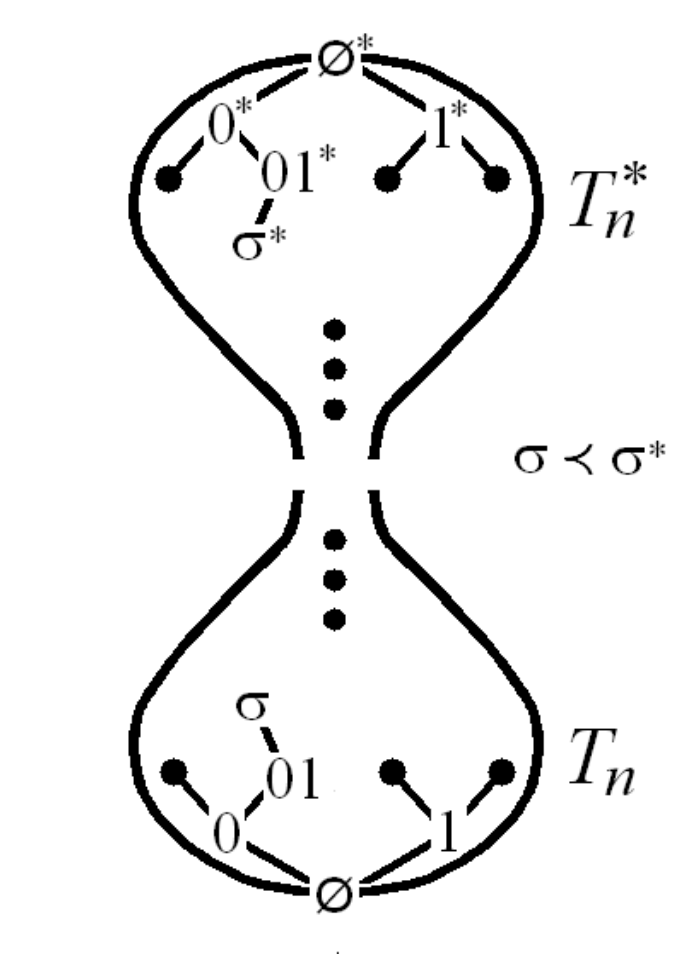}
				\caption{The lattice $L_n$ from Proposition \ref{completeHard}.}
				\label{fig:Ln}
			\end{figure}
		\subsection{Compact elements of a lattice can be \texorpdfstring{$\Pi^{1}_{1}$}{Pi11}-complete}

			\begin{df}
				An element $a\in L$ is \emph{compact} if for each subset $S\subseteq L$, if $a\Le \sup S$ then there is a finite subset $S'\subseteq S$ such that $a\Le \sup S'$.
				Thus, if $a\Le\sup S$ but for each finite subset $S'\subseteq S$, $a\not\Le \sup S'$, then $S$ is a \emph{witness} for the non-compactness of $a$.
			\end{df}

			\begin{lem}
				In each computable lattice $L$, the set of compact elements of $L$ is $\Pi^1_1$. 
			\end{lem}
			\begin{proof}
				Similarly to the situation in Lemma \ref{1}, the statement that $a$ is compact consist of a universal set quantifier over $S$ followed by an arithmetical matrix.
			\end{proof}

			\begin{exa}\label{compactExa}
				Let $L[a]=\omega+1\cup\{a\}$ be ordered by $0\prec a\prec\omega$, and let the element $a$ be incomparable with the positive numbers.
				Then $a$ is not compact, because $a\Le\sup\omega$ but $a\not\Le\sup S'$ for any finite $S'\subseteq \omega$.
			\end{exa}

			\begin{figure}[htb!]
				\centering%
				\includegraphics[height=7.5cm]{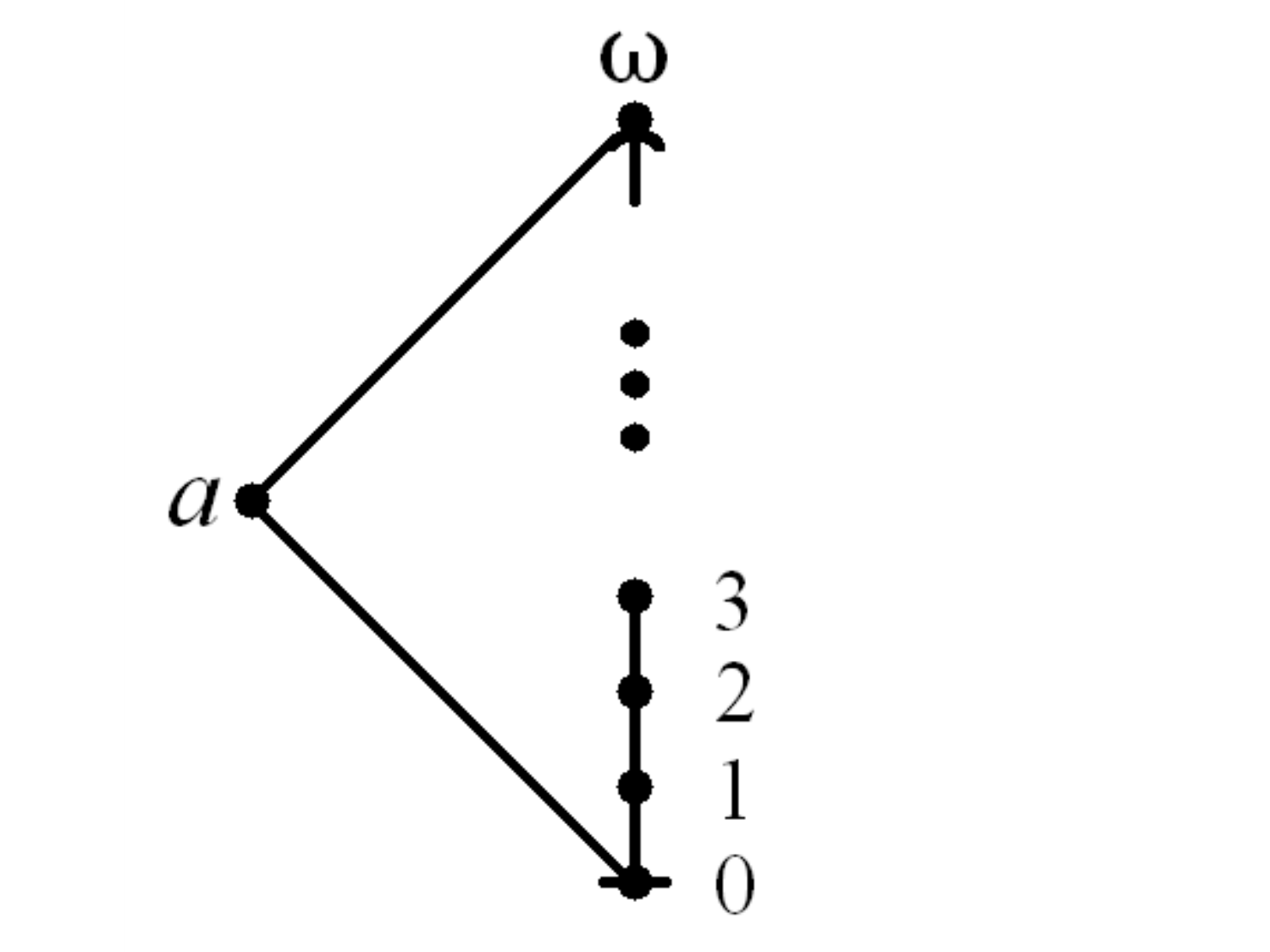}
				\caption{The lattice $L[a]$ from Example \ref{compactExa}.}
				\label{fig:La}
			\end{figure}

			\begin{df}
				A lattice $\,\,(L,\Le)\,\,$ is \emph{compactly generated} if every element is the supremum of a set of compact elements.
			 	A lattice is \emph{algebraic} if it is complete and compactly generated.
			\end{df}
			\begin{pro}\label{compactHard}
				There is a computable complete lattice $L$ such that the set of compact elements of $L$ is $\Pi^1_1$-hard.
				Moreover, $L$ is not algebraic.
			\end{pro}
			\begin{proof}
				Let $L$ consist of disjoint copies of the trees $T_n$, $n\in\omega$, each having the prefix ordering; least and greatest elements $0$ and $1$;
				and elements $a_n$, $n\in\omega$, such that $\sigma\prec a_n$ for each $\sigma\in T_n$, and $a_n$ is incomparable with any element not in $T_n\cup\{0,1\}$ (see Figure \ref{fig:L}).

				Suppose $T_n$ has an infinite path $S$. Then $a_n=\sup S$ but $a_n\not\Le \sup S'$ for any finite $S'\subseteq S$, since $\sup S'$ is rather an element of $S$. Thus $a_n$ is not compact. 

				Conversely, suppose $T_n$ has no infinite path, and $a_n\Le\sup S$ for some set $S\subseteq L$.
				If $S$ contains elements from $T_m\cup\{a_m\}$ for at least two distinct values of $m$, say $m_1\ne m_2$,
				then $\sup S=1=\sigma_1\vee\sigma_2$ for some $\sigma_i\in S\cap(T_{m_i}\cup \{a_{m_i}\})$, $i=1, 2$.
				So $a_n\Le\sup S'$ for some $S'\subseteq S$ of size two. If $S$ contains 1, there is nothing to prove.
				The remaining case is where $S$ is contained in $T_m\cup\{a_m,0\}$ for some $m$. Since $a_n\Le\sup S$, it must be that $m=n$.
				If $S$ is finite or contains $a_n$, there is nothing to prove.
				So suppose $S$ is infinite. Since $T_n$ has no infinite path, there must be two incomparable elements of $T_n$ in $S$.
				Their join is then $a_n$, since $T_n$ is a tree, and so $a_n\Le\sup S'$ for some $S'\subseteq S$ of size two. 

				Thus we have shown that $a_n$ is compact if and only if $T_n$ has no infinite path.
				There is a computable presentation of $L$ where $a_n$ is a computable function of $n$, for instance we could let $a_n=2n$.
				Thus letting $f(n)=2n$, we have that $T_n$ has no infinite path iff $f(n)$ is compact, i.e. $\{a\in L: a\text{ is compact}\}$ is $\Pi^1_1$-hard.

				It remains to show that $L$ is not algebraic.
				Fix $n$ such that $T_n$ has an infinite path $P$, and also some nontrivial finite paths that do not extend to infinite paths. Let $\sigma$ be on such a finite path. Then each element of $P$ is compact. However, $\sigma$ is below the supremum of $P$, but not below any join of finitely many elements of $P$, so $\sigma$ is not compact. Moreover, $\sigma$ is join irreducible, being located on the tree $T_n$. Thus $\sigma$ is not a join of compact elements below it, and so $L$ is not compactly generated.
			\end{proof}
			From the proof of Proposition \ref{compactHard} we obtain the following corollary.
			\begin{cor}[$\RCA_0$]
				The following principle is equivalent to $\PCA_0$:
				``For each countable lattice $L$, there is a set consisting of exactly the compact elements of $L$.''
			\end{cor}

			\begin{figure}[htb!]
				\centering%
				\includegraphics[height=6.5cm]{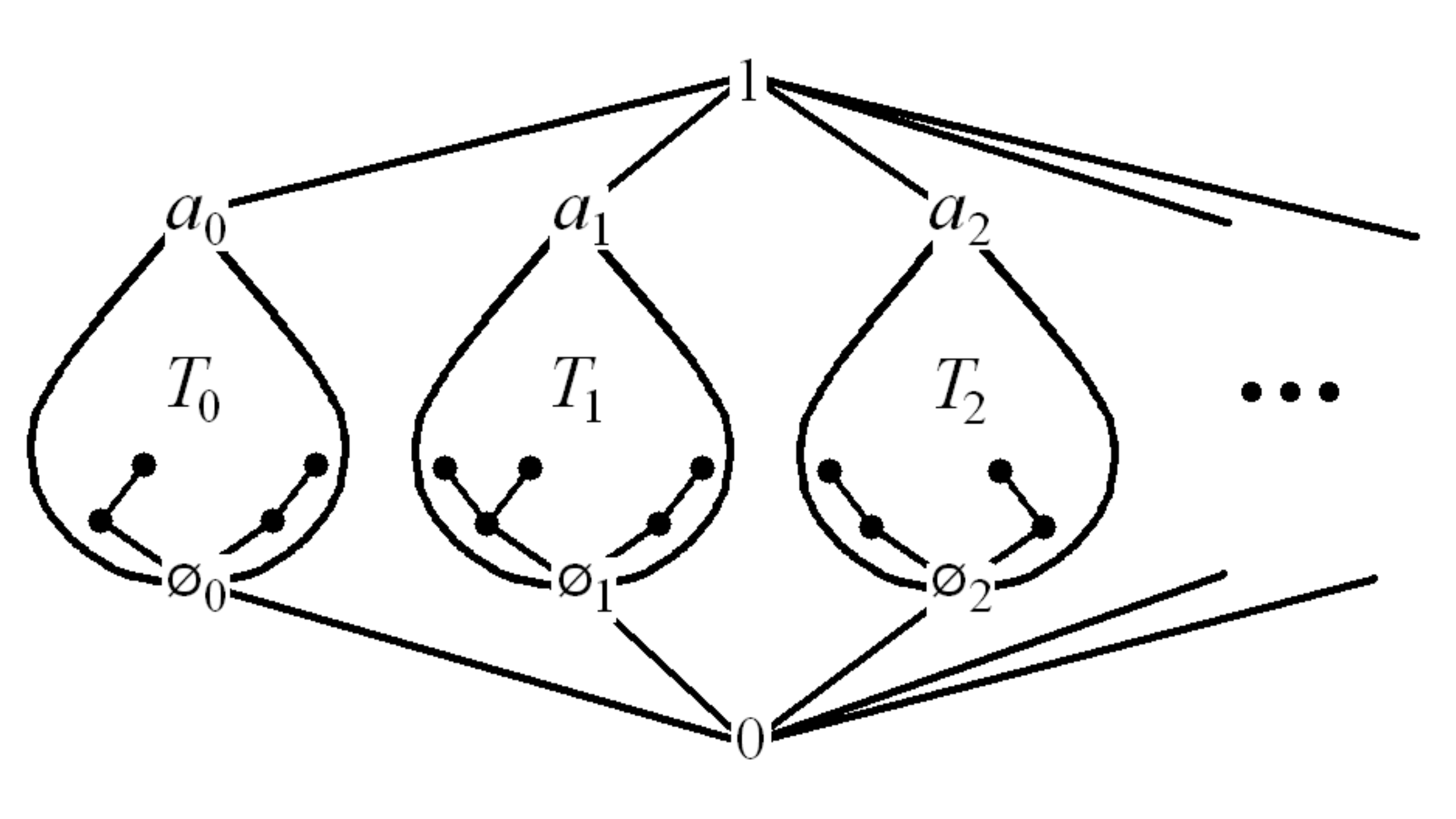}
				\caption{The lattice $L$ from Proposition \ref{compactHard}.}
				\label{fig:L}
			\end{figure}
			\begin{question}
				Is there a computable algebraic lattice such that the set of its compact elements is $\Pi^1_1$-complete?
			\end{question}
		\subsection{Index set of algebraic lattices is \texorpdfstring{$\Pi^{1}_{1}$}{Pi11}-complete}
			\begin{lem}
				The set of indices of computable lattices that are algebraic is $\Pi^1_1$.
			\end{lem}
			\begin{proof}
				Let $L$ be a computable lattice and $C$ the set of its compact elements. $L$ is algebraic if it is complete (this property is $\Pi^1_1$ by Lemma \ref{1}) and \emph{each element is the supremum of its compact predecessors}, i.e.,
				any element that is above all the compact elements below $a$ is above $a$:
				\[
					\forall a(\forall b(\forall c(c\in C\And c\Le a\to c\Le b)\to a\Le b))
				\]
				Equivalently,
				\[
					\forall a(\forall b(\exists c(c\in C\And c\Le a\And c\not\Le b)\text{ or } a\Le b))
				\]
				This is equivalent to a $\Pi^1_1$ statement since, by the Axiom of Choice, any statement of the form
				$\exists c\,\,\forall S\,\, A(c,S)$ is equivalent to $\forall (S_c)_{c\in\omega}\,\, \exists c \,\,A(c,S_c)$.
			\end{proof}

			\begin{exa} 
				The lattice $(\omega+1,\le)$ is compactly generated, since the only noncompact element $\omega$ satisfies $\omega=\sup\omega$.
				The lattice $L[a]$ from Example \ref{compactExa} and Figure \ref{fig:La} is not compactly generated, as the noncompact element $a$ is not the supremum of $\{0\}$.
			\end{exa}

			\begin{pro}\label{algebraicHard}
				The set of indices of computable lattices that are algebraic is $\Pi^1_1$-hard.
			\end{pro}
			\begin{proof}
				Let the lattice $T_n[a]$ consist of $T_n$ with the prefix ordering, and additional elements $0\prec a\prec 1$ such that $a$ is incomparable with each $\sigma\in T_n$,
				and $0$ and $1$ are the least and greatest elements of the lattice.
				Note that $T_n[a]$ is always complete, since any infinite set has supremum equal to $1$. We claim that $T_n[a]$ is algebraic iff $T_n$ has no infinite path.

				Suppose $T_n$ has an infinite path $S$. Then $a\Le \sup S$, but $a\not\Le \sup S'$ for any finite $S'\subseteq S$.
				Thus $a$ is not compact, and so $a$ is not the sup of its compact predecessors ($0$ being its only compact predecessor), which means that $T_n[a]$ is not an algebraic lattice.

				Conversely, suppose $T_n[a]$ is not algebraic. Then some element of $T_n[a]$ is not the join of its compact predecessors.
				In particular, some element of $T_n[a]$ is not compact. So there exists a set $S\subseteq T_n[a]$ such that for all finite subsets $S'\subseteq S$, $\sup S'<\sup S$.
				In particular $S$ is infinite. Since each element except 1 has only finitely many predecessors, we have $\sup S=1$.
				Notice that $T_n[a] \backslash \{1\}$ is actually a tree, so if $S$ contains two incomparable elements then their join is already 1, contradicting the defining property of $S$.
				Thus $S$ is linearly ordered, and infinite, which implies that $T_n$ has an infinite path.
			\end{proof}

			\begin{figure}[htb!]
				\centering%
				\includegraphics[height=7.5cm]{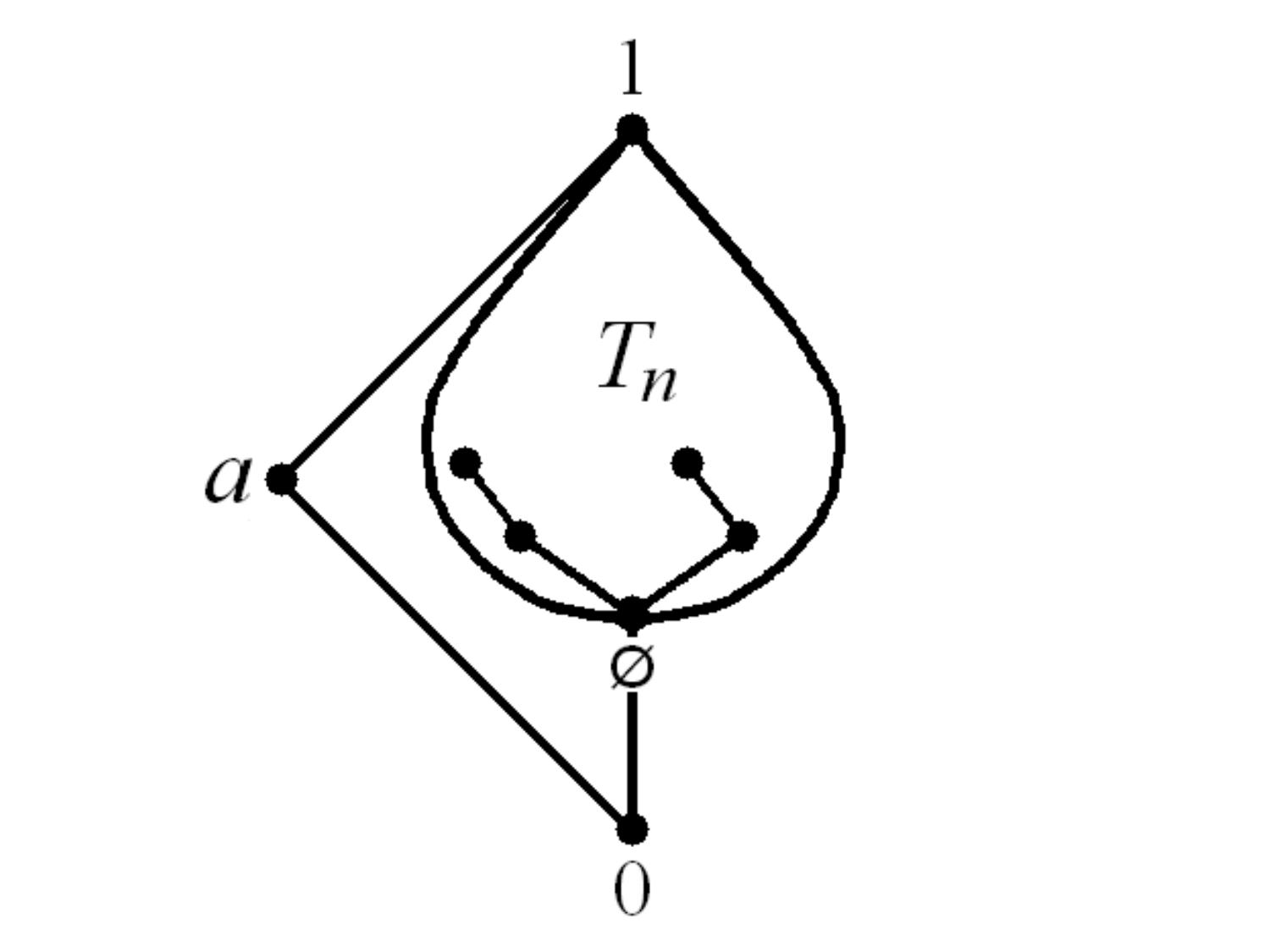}
				\caption{The lattice $T_n[a]$ from Proposition \ref{algebraicHard}.}
				\label{fig:Tna}
			\end{figure}

	\section{Lattices of equivalence relations}

		Let $\Eq(A)$ denote the set of all equivalence relations on $A$. Ordered by inclusion, $\Eq(A)$ is a complete lattice.
		In a sublattice $L\subseteq\Eq(A)$, we write $\sup_L$ for the supremum in $L$ when it exists, and $\sup$ for the supremum in $\Eq(A)$, and note that $\sup\le\sup_L$. 

		A \emph{complete sublattice of $\Eq(A)$} is a sublattice $L$ of $\Eq(A)$ such that $\sup_L=\sup$ and $\inf_L=\inf$.
		A sublattice of Eq$(A)$ that is a complete lattice is not necessarily a complete sublattice in this sense. The following lemma is well known.
		A good reference for lattice theory is the monograph of Gr\"atzer \cite{MR2451139}.

		\begin{lem}\label{referi}
			Suppose $A$ is a set and $(L,\subseteq)$ is a complete sublattice of $\Eq(A)$. Then an equivalence relation $E$ in $L$ is a compact member of $L$ if and only if $E$ is finitely generated in $L$.
		\end{lem}
		\begin{proof}
			One direction only uses that $L$ is a sublattice of Eq$(A)$ and $L$ is complete as a lattice.
			Suppose $E$ is not finitely generated in $L$. Let $C_{(a,b)}$ denote the infimum of all equivalence relations in $L$ that contain $(a,b)$.
			Then $E \subseteq\sup_L\{C_{(a,b)}: aEb\}$, but $E$ is not below any finite join of the relations $C_{(a,b)}$. So $E$ is not compact.

			Suppose $E$ is finitely generated in $L$. So there exists an $n$ and pairs $(a_1,b_1)$,$\ldots$, $(a_n,b_n)$ such that $a_iEb_i$ for all $1\le i\le n$, and
			for all equivalence relations $F$ in $L$, if $a_iFb_i$ for all $1\le i\le n$ then $E\subseteq F$.
			Suppose $E\subseteq \sup_L\{E_i: 1\le i<\infty\}$ for some $E_1,E_2,\ldots\in L$. Since $L$ is a complete sublattice of Eq$(A)$, $\sup_L=\sup$, so $E\subseteq \sup\{E_i: 1\le i<\infty\}$.
			Note that $\sup\{E_i: 1\le i<\infty\}$ is the equivalence relation generated by the relations $E_i$ under transitive closure.
			So there is some $j=j_n<\infty$ such that $\{(a_i,b_i):1\le i\le n\}\subseteq \bigcup_{i=1}^j E_i$ and hence $E\subseteq\bigcup_{i=1}^j E_i$. Thus $E$ is compact.
		 \end{proof}

		A \emph{computable complete sublattice of $\Eq(\omega)$} is a uniformly computable collection $\mc E=\{E_i\}_{i\in\omega}$ of distinct equivalence relations on $\omega$ such that
		$(\mc E,\subseteq)$ is a complete sublattice of $\Eq(\omega)$.
		We say that the lattice $L=(\omega,\Le)$ is \emph{computably isomorphic} to $(\mc E, \subseteq)$ if
		there is a computable function $\varphi:\omega\to\omega$ such that for all $i$, $j$, we have $i\Le j \iff E_{\varphi(i)}\subseteq E_{\varphi(j)}$.

		\begin{lem}\label{ccc}
			The indices of compact congruences in a computable complete sublattice of $\Eq(\omega)$ form a $\Sigma^0_2$ set.
		\end{lem}
		\begin{proof}
			Suppose the complete sublattice is $\mc E=\{E_i\}_{i\in\omega}$. By Lemma \ref{referi}, $E_k$ is compact if and only if it is finitely generated, i.e., 
			\[
				\exists n\,\, \exists a_1,\ldots,a_n \,\,\exists b_1,\ldots,b_n\,\,\left[ \bigwedge_{i=1}^n a_iE_k b_i\And \forall j \left(\bigwedge_{i=1}^n a_iE_j b_i \to E_k\subseteq E_j\right)\right].
			\]
			Here $E_k\subseteq E_j$ is $\Pi^0_1$: $\forall x\forall y\,\, (xE_k y\to xE_j y)$, so the formula is $\Sigma^0_2$.
		\end{proof}

		\subsection{Congruence lattices}
			An \emph{algebra} $\mf A$ consists of a set $A$ and functions $f_i:A^{n_i}\to A$. Here $i$ is taken from an index set $I$ which may be finite or infinite, and $n_i$ is the arity of $f_i$.
			Thus, an algebra is a purely functional model-theoretic structure.
			A \emph{congruence relation} of $\mf A$ is an equivalence relation on $A$ such that for each unary $f_i$ and all $x, y\in A$, if $xEy$ then $f_i(x)Ef_i(y)$,
			and the natural similar property holds for $f_i$ of arity greater than one. 

			The congruence relations of $\mf A$ form a lattice under the inclusion (refinement) ordering. This lattice $\text{Con}(\mf A)$ is called the \emph{congruence lattice} of $\mf A$.

			The following lemma is well-known and straight-forward.
			\begin{lem}\label{straight}
				If $\mf A$ is an algebra on $A$, then $\text{Con}(\mf A)$ is a complete sublattice of $\Eq(A)$.
			\end{lem}
			\begin{thm}[Gr\"atzer-Schmidt \cite{MR0151406}]\label{GS}
				Each algebraic lattice is isomorphic to the congruence lattice of an algebra. 
			\end{thm}

			\begin{rem}
				Let $A$ be a set, and let $L$ be a complete sublattice of $\Eq(A)$.
				Then $L$ is algebraic \cite{MR2451139}, and so by Theorem \ref{GS} $L$ is \emph{isomorphic} to $\text{Con}(\mf A)$ for some algebra $\mf A$ on some set,
				but it is not in general possible to find $\mf A$ such that $L$ is \emph{equal} to $\text{Con}(\mf A)$. In fact, it suffices to take any finite lattice table that is not Malcev homogeneous in the sense of Definition 3.1 of \cite{Kjos-Hanssen:03}.
			\end{rem}

		\subsection{Principal congruences can be Turing complete}
				Let $\mf A$ be an algebra. The least congruence relation $\sim$ on $\mf A$ with $a\sim b$ is denoted by $C_{\mf A}(a,b)$
				and is called the principal congruence relation generated by the pair $(a,b)$.

				\begin{df}
					We say that the algebra $\mf A=\{f_n\mid n\in\omega\}$ is computable if the set
					\[
						\{\la \la x_1,\dots,x_k\rangle,y,n\ra\mid f_n(x_1,\dots,x_k)=y\}
					\]
					is computable.
				\end{df}
				\begin{thm}
					There is a computable algebra $\mf A$ and $a,b\in A$ such that the Turing degree of $C_{\mf A}(a,b)$ is $0'$.
				\end{thm}
				\begin{proof}
					Let $0'=\{g(n)\mid n\in\omega\}$ where $g$ is computable, and let the operations of $\mf A$ be unary functions $\{f_s\}_{s\in\omega}$.
					Let $f_s(a_0)=a_{g(s)}$ and $f_s(b_0)=b_{g(s)}$,
					where $A=\{a_n\mid n\in\omega\}\cup\{b_n\mid n\in\omega\}$, a union of two disjoint infinite sets; let $f_s$ be the identity on $A\backslash\{a_0,b_0\}$.
					Then for $k>0$, $(a_k,b_k)\in C_{\mf A}(a_0,b_0)$ iff $k\in 0'$. So we can let $(a,b) = (a_0, b_0)$.
				\end{proof}

	\section{Reverse mathematics}
		We consider the following standard axiom systems of reverse mathematics \cite{Simpson}:
		\begin{itemize}
			\item $\RCA_{0}$ (recursive comprehension axiom);
			\item $\ACA_{0}$ (arithmetical comprehension axiom);
			\item $\PCA_{0}$ ($\Pi^{1}_{1}$-comprehension axiom);
			\item $\WKL_{0}$ (weak K\"onig's lemma);
			\item $\RT^{2}_{2}$ (Ramsey's theorem for pairs).
		\end{itemize}

		\begin{df}
			The axiom system $\GS$ (Gr\"atzer-Schmidt) consists of $\RCA_{0}$ plus the following axiom:
			For each algebraic lattice $L$ there exists
				\begin{enumerate}
					\item an algebra $\mf A$,
					\item a set $\{E_i\}_{i\in\omega}$ of congruences of $\mf A$ such that each congruence of $\mf A$ is one of the $E_i$, and
					\item an isomorphism $\varphi$ between $L$ and $\{E_i\}_{i\in\omega}$.
				\end{enumerate}
		\end{df}
		\begin{rem}
			For this theorem to fall within the scope of reverse mathematics, for each countable lattice $L$, there must exist a \emph{countable} algebra $\mf A$ satisfying the properties above. That this is the case can be seen from Pudl\'ak's proof \cite{MR0429699} of the Gr\"atzer-Schmidt theorem, which we discuss in more detail below.
		\end{rem}

		\begin{df}
			Let $\GSD$ be the Gr\"atzer-Schmidt theorem for distributive lattices: \emph{every distributive algebraic lattice is isomorphic to the congruence lattice of an algebra}.
		\end{df}

	\section{Compact elements in algebraic lattices of restricted kinds}
		\subsection{Distributive lattices}
			As a contrast to the case of arbitrary lattices (Proposition \ref{compactHard}), in the distributive case the complexity of the set of compact elements reduces from $\Pi^{1}_{1}$ to $\Pi^{0}_{3}$ (Theorem \ref{distr3}). This is also sharp (Theorem \ref{d3hard}), which will enable us to show that $\WKL_0+\RT^2_2$ does not imply $\GSD$ (Corollary \ref{enabled}). We first need a proposition.
			\begin{proposition}[$\ACA_0$]\label{placed}
				If $L$ is a countable algebraic lattice and $a\in L$ is not compact then there is a witness $C\subseteq \{x\mid x<a\}$.
				Moreover, we can assume that $C=\{c_{i}\mid i\in \omega \}$ where the $c_{i}$ are strictly increasing. 
			\end{proposition}
			\begin{proof}
				Let $C=\{d_{i}\}$ witness the fact that $a$ is not compact. Thus $a\le\sup C$ but for each finite $C'\subset C$, $a\not\le\sup C'$.
				By closing under finite joins of initial segments and thinning out the sequence, we can assume that the $d_{i}$ are strictly increasing.

				As $L$ is algebraic, $a$ is the join of the compact elements $\le a$.
				Since moreover $a$ is not itself compact, $a$ is the join of the compact elements $c<a$.

				Since $a\le\sup_i d_i$, each compact $c\le a$ is below some $d_0\vee\dots\vee d_i = d_i$, and hence $c\le d_i\wedge a < a$.

				Thus $a=\bigvee_i (d_{i}\wedge a)$.
				Finally, let $\{c_i\}_{i\in\omega}$ be a strictly increasing subsequence of the sequence $\{d_{i}\wedge a\}_{i\in\omega}$.
			\end{proof}
			\begin{df}
				We say that \emph{$b$ is a coatom relative to $a$}, written $b\sqsubset a$, if 
				\[
					b<a\quad\text{and}\quad\lnot \exists y(b<y<a).
				\]
			\end{df}

			\begin{theorem}[$\ACA_0$]\label{distr3}
				In an algebraic countable distributive lattice $L$, the set $\{a\in L\mid a\text{ is compact}\}$ has the $\Pi _{3}^{0}(L)$ form
				\[
					\{a\in L\mid (\forall x<a)(\exists b)(x\le b\sqsubset a)\}.
				\]
			\end{theorem}

			\begin{proof}
				Fix $a\in L$. Let $B=\{b_{j}\}=\{b\mid b\sqsubset a\}$. We must show that $a$ is compact if and only if
				\[
					(\forall x<a)(\exists b)(x\le b\sqsubset a).
				\]
				\noindent\emph{Only if} direction: Assume that there is an $z<a$ with no $b\in B$ above it. Let 
				\[
					D=\{x<a\mid x\text{ is not below any }b\in B\}=\{d_{i}\}.
				\]
				Note that $D$ is nonempty by assumption and has no maximal elements by definition. We build an increasing sequence $c_{j}\in D$ such that for each $i$, $d_{i}\neq\vee c_{j}$.
				Again by our assumptions this guarantees that $\vee c_{j}=a$ as required to show that it is not compact. Let $c_{0}=z$ and suppose we have defined $c_{k}$.
				We want to choose $c_{k+1}>c_{k}$ in $D$ so as to guarantee that $d_{k}$ will not be the join of all the $c_{j}$.
				If $d_{k}\ngeq c_{k}$ then $d_{k}$ cannot be the join of the $c_{j}$ and we can take any $c\in D$ with $c>c_{k}$ as $c_{k+1}$.
				If $d_{k}\geq c_{k}$ we can take any $c\in D$ with $c>d_{k}$ as once again we have guaranteed that $d_{k}\neq \vee c_{j}$.

				\noindent \emph{If} direction: We suppose that every $x<a$ is below some $b\in B$ and, for the sake of a contradiction, that $a$ is not compact.
				Then by Proposition \ref{placed}, some $C=\{c_{i}\}$ (a strictly increasing sequence of elements below $a$) witnesses that $a$ is not compact.
				If $\exists j\forall i(c_{i}\leq b_{j})$ then $\vee c_{i}\leq b_{j}<a$ for any such $j$ contradicting our choice of $C$. Thus $\forall j\exists i(c_{i}\nleq b_{j})$.
				If $B$ is finite, there is an $i$ such that $\forall j(c_{i}\nleq b_{j})$ as the $c_{i}$ are increasing. This would contradict our case assumption.

				Finally, we suppose that $B$ is infinite. We build a nondecreasing sequence $d_{n}$ of elements strictly below $a$ with $d_{0}=c_{0}$ which has no join in $L$ below $a$
				for a contradiction to the completeness of $L$.
				Each $d_{k+1}$ will be of the form $b_{j_{1}}\wedge b_{j_{2}}\wedge \ldots \wedge b_{j_{k}}\wedge c_{l_{k}}$
				and its choice will guarantee that $x_{k}$ is not the join of all the $d_{n}$ where $L=\{x_{k}\}$.

				Suppose we have $d_{k}$ and want to define $d_{k+1}$. First ask if $(\exists b\in B)(b\geq d_{k}~\&~b\ngeq x_{k})$. If so, we let $b_{j_{k+1}}$ be such a $b$ and $l_{k+1}=l_{k}$.
				In this case $d_{k+1}=d_{k}$ and, by the intended form of our $d_{n}$,
				we have guaranteed that $b\geq d_{n}$ for every $n$ and so that $b\geq \vee d_{n}.$ As $b\ngeq x_{k}$, $x_{k}\neq \vee d_{n}$ as required.
				Otherwise, for every $b\in B$ with $b\geq c_{l_{k}}$, $b\geq x_{k}$. Choose one such $b$ not equal to any $b_{j_{m}}$, $m\leq k$, and a $p>l_{k} $ such that $b\ngeq c_{p}$.

				Note that $\{j\mid b_{j}\geq c_{i}\}$ is nonempty for every $i$ by our case assumption.
				Thus $\forall i\exists ^{\infty }j(b_{j}\geq c_{i})$ since otherwise (as the $c_{i}$ are increasing) there would be a finite set $F$ such that $\forall i\forall j\in F(b_{j}\geq c_{i})$ and so
				$\vee c_{i}\leq\wedge \{b_{j}\mid j\in F\}<a$ contradicting our choice of $C$.
				Also note that $\forall n\exists i\forall j\geq i(c_{j}\nleq b_{n})$ as otherwise $\forall j(c_{j}\leq b_{n})$ and so $\vee c_{i}\leq b_{n}$ again contradicting our choice of $C$.
		 
				Now let $j_{k+1}=j_{k}$ and $c_{l_{k+1}}=c_{p}$. As $c_{p}\geq c_{l_{k}}$, $d_{k+1}\geq d_{k}$. As $b\geq c_{l_{k}}$, $b\geq x_{k}$.
				On the other hand, $b$ is not any of the $b_{j_{m}}$ for $m\leq k+1$ and so is not above any of them. Moreover, it is not above $c_{p}=c_{l_{k+1}}$. Thus it is not above 
				\[
					d_{k+1}=b_{j_{1}}\wedge \ldots \wedge b_{j_{k+1}}\wedge c_{l_{k+1}}
				\]
				by distributivity, as we now show:

				As $b\in B$, $b\vee b_{j_{m}}=a=c_{l_{k+1}}\vee b$ for $m\leq k+1$. But if 
				\[
					b\geq b_{j_{1}}\wedge \ldots \wedge b_{j_{k+1}}\wedge c_{l_{k+1}}
				\]
				then 
				\[
					b=\left(\left(\bigwedge_{i=1}^{k+1} b_{j_{i}}\right)\wedge c_{l_{k+1}}\right)\vee b=(b\vee b_{j_{1}})\wedge \ldots \wedge (b\vee b_{j_{k+1}})\wedge (b\vee c_{l_{k+1}})
				\]
				but as $b\in B$ each of these terms
				(and so their join) is equal to $a$ for the desired contradiction. Thus $x\neq \vee d_{n}$ as required.
			\end{proof}

			\begin{proposition}[Folklore]
				For every $\Pi^0_3$ predicate $P$, there is a computable function $h(x, y)$ such that for all $x$ and $y$, $W_{h(x, y)}$ is an initial segment of $\omega$, and \[P(x) \Rightarrow (\forall y)(W_{h(x, y)} \text{ is finite})\] and \[\neg P(x) \Rightarrow (\exists! y)(W_{h(x, y)} = \omega). \]
			\end{proposition}
			\begin{proof}
				It is well-known (see, for example, Soare \cite{SoareBook}, Theorem 4.3.4) that there is a function $g(x, y)$ such that \[P(x) \Leftrightarrow (\forall y)(W_{g(x, y)} \text{ is finite}).\] We describe a uniform sequence $\{C_i\}_{i \in \omega}$ of c.e.\ sets. At each stage $s$ of the enumeration of this sequence, for each $y \le s$, there is a designated ``destination'' $i_{y, s} \in \omega$ for $W_{g(x, y)}$. By a ``new destination'', we mean the least $n \in \omega$ that has not yet been used as a destination.

				 At stage $s$, choose a new destination $i_{s, s}$. If it exists, let $z < s$ be the least such that a new element has just entered $W_{g(x, z)}$. Then 
					\begin{itemize}
						\item enumerate into $C_{i_{z, s}}$ the least element not already in it, and
						\item choose new destinations for $W_{g(x, y)}$ for all $y$ such that $z < y \le s$.
					\end{itemize}
				This describes the enumeration of $\{C_i\}_{i \in \omega}$.

				We verify that this sequence has the desired properties. If there is a $y$ such that $W_{g(x, y)}$ is infinite, then let it be the least such. After some stage $s$, new elements will cease to appear in $W_{g(x, y')}$ for $y' < y$, and $i_{y, s}$ will never again be redefined. Thus $C_{i_{y, s}} = \omega$. If $j \ne i_{y, s}$ is ever a destination for some $W_{g(x, z)}$ for some $z > y$, it will cease to be so when a new element is enumerated into $W_{g(x, y)}$, hence $C_j$ will be finite. On the other hand, if $W_{g(x, y)}$ is finite for all $y$, $C_j$ is finite for all $j$, since each such $j$ is ever a destination for $W_{g(x, y)}$ for exactly one value of $y$.

				Finally, let $W_{h(x, y)} = C_y$.
			\end{proof}

			\begin{thm}\label{d3hard}
				There is a computable distributive algebraic lattice for which the set of compact elements is complete $\Pi _{3}^{0}$. 
			\end{thm}
			\begin{proof}
				Given a complete $\Pi _{3}^{0}$ set $P$, let $h$ be as in the proposition above. Our lattice shall contain elements $a_{i}$ for each $i<\omega$, and elements $a_{i,j,k}$ for each triple $(i,j,k)$ such that $ k \in W_{h(i, j)}$.
				The plan is that $a_{i}$ will be compact iff $P(i)$ holds. Let \[\alpha_i = \{a_{i, j, k} \mid k \in W_{h(i, j)} \text{ and } j \in \omega\},\] and $\Lambda = \{0, 1\} \cup \bigcup_i \{a_i\} \cup \bigcup_i \alpha_i$.

				The ordering among the elements of $\Lambda$ is specified by 
				\begin{eqnarray*}
					a_{i,j,k}\leq a_{\hat{\imath},\hat{\jmath},\hat{k}}~	&\quad\Longleftrightarrow\quad	& i=\hat{\imath}~\&~j=\hat{\jmath}~\&~k\leq \hat{k};\\
					a_{i}\leq a_{\hat{\imath}}						&\quad\Longleftrightarrow\quad	& i=\hat{\imath}; \\
					a_{i,j,k}\leq a_{\hat{\imath}}					&\quad\Longleftrightarrow\quad	& i=\hat{\imath};
				\end{eqnarray*}
				and no $a_{\hat{\imath}}$ is below any $a_{i,j,k}$. The top element $1$ is above all others in $\Lambda$, while $0$ is below.
				
				This determines joins betweens pairs of elements of $\Lambda$:
				\begin{align*}
					a_{i}\vee a_{\hat{\imath}} & = 1 \text{ for $i\neq \hat{\imath}$}\\
					a_{\hat{\imath}}\vee a_{i,j,k} & = \begin{cases}
															1 &\text{ for $i\neq \hat{\imath}$ and} \\ 
															a_{\hat{\imath}} & \text{ if $i=\hat{\imath}$;}
														\end{cases}\\
					a_{i,j,k}\vee a_{\hat{\imath},\hat{\jmath},\hat{k}} & = \begin{cases}
																				1 & \text{if $i\neq \hat{\imath}$; } \\
																				a_{i} & \text{ if $i=\hat{\imath}$ and $j\neq \hat{\jmath}$; and } \\
																				a_{i,j,k} & \text{ if $i=\hat{\imath}$, $j=\hat{\jmath}$ and $\hat{k}\leq k$. } \\
					\end{cases}\\
				\end{align*}
				
				These relations extend to arbitrary joins as follows: Let $\Lambda_0 \subseteq \Lambda$. If $\Lambda_0$ contains a pair that join up to $1$ then $\bigvee \Lambda_0 = 1$. Otherwise, all elements of $\Lambda_0$ have the same $i$.
				If there are two with different $j$ (or $a_{i}$ itself occurs) then $\bigvee \Lambda_0 = a_i$. Otherwise, they are all of the form $a_{i,j,k}$ for a fixed $i$ and $j$. If $\sup\{k \mid a_{i, j, k} \in \Lambda_0\} = \omega$, then $\bigvee \Lambda_0$ is again $a_{i}$. If it is $\hat{k} \in \omega$, then $\bigvee \Lambda_0 = a_{i,j,\hat{k}}$. Thus, $\Lambda$ is closed under arbitrary joins.

				To each element $x$ in $\Lambda$, we now associate a subset $\Gamma(x)$ of $\omega$. Let $B$ and $C$ be infinite uniformly computable sets such that $B \cup C = \omega$, and let $\{B_i\}_{i \in \omega}$ and $\{C_i\}_{i \in \omega}$ be partitions of $B$ and $C$, respectively, into infinite computable pairwise disjoint sets. Let $f_i: \omega^2 \rightarrow C_i$ be a uniform family of computable bijections. Now let
				\begin{eqnarray*}
					\Gamma(0) 			& = & \emptyset \\
					\Gamma(1) 			& = & \omega    \\
					\Gamma(a_i) 			& = & A_i = \omega \setminus B_i \\
					\Gamma(a_{i, j, k}) 	& = & A_{i, j, k} = \omega \setminus (B_i \cup \{f_i(j, \hat{k}) \mid \hat{k} > k\}).
				\end{eqnarray*}
				The following claims are easily verified:
				\begin{claim}
					For all $x, y \in \Lambda$, $x \le y \Leftrightarrow \Gamma(x) \subseteq \Gamma(y)$.
				\end{claim}
				\begin{claim}
					For any $\Lambda_0 \subseteq \Lambda$, $\Gamma(\bigvee \Lambda_0) = \bigcup_{x \in \Lambda_0} \Gamma(x)$.
				\end{claim}

				Let $L$ be the collection of sets obtained by closing the image of $\Gamma$ under finite intersections. The distributivity of union and intersection ensures that $L$ is also closed under finite unions. Thus $L$ is a distributive lattice, and its order extends the ordering on $\Lambda$ (which we identify with its image under $\Gamma$). The domain of our computable presentation of $L$ will be $\omega$: we can assume that only finitely many elements are enumerated into the uniformly c.e. sequence $W_{h(i, j)}$ at every stage, and for each such element and the finitely many new intersections it gives rise to, we allocate as yet unused natural numbers while ensuring that every natural number will be eventually allocated. We must now verify that 

				\begin{enumerate}
					\item the relations and operations on $L$ are computable, and
					\item $L$ is algebraic.
				\end{enumerate}

				To this end, we derive a normal form for the finite meets making up the lattice. Suppose $x \in L$ is neither $0$ nor $1$. It is an intersection of finitely many elements of the form $A_{i}$ and the form $A_{i,j,k}$.
				For each $i$, if any $A_{i,j,k}$ appears, we may eliminate all terms of the form $A_{i}$ and $A_{i,j,\hat{k}}$ except for the smallest $\hat{k}$ so occurring. We now have a normal form of $x$ given by
				\[
					x=\bigwedge _{i\in F}A_{i}\wedge \bigwedge _{i\in G}\bigwedge _{j\in G_{i}}A_{i,j,k_{i,j}}
				\]
				with $F,G,G_{i}$ finite nonempty sets and $F$ and $G$ disjoint. We say that \emph{$x$ is represented by $\left\langle F,G,\left\langle G_{i}\mid i\in G\right\rangle\right\rangle $}. It can be verified that the normal form representation of an element is unique.

				\begin{claim}		
					$L$ is computable as a lattice.
				\end{claim}		
				\begin{proof}
				Suppose we have 
				\[
					x=\bigwedge _{i\in F}A_{i}\wedge \bigwedge_{i\in G}\bigwedge _{j\in G_{i}} A_{i,j,k_{i,j}}\quad\text{and}\quad
					\hat{x}=\bigwedge _{i\in \hat{F}}A_{i}\wedge \bigwedge _{i\in \hat{G}}\bigwedge _{j\in \hat{G}_{i}} A_{i,j,\hat{k}_{i,j}}.
				\]
				We claim that 
				\[
					\tag{$\ast$}
					x\leq \hat{x}\quad\Longleftrightarrow\quad
					~\hat{F}\setminus F\subseteq G
					~\&~\hat{G}\subseteq G
					~\&~(\forall i\in \hat{G})(\hat{G}_{i}\subseteq G_{i})
					~\&~(\forall i\in \hat{G})(\forall j\in \hat{G}_{i})(\hat{k}_{i,j}\geq k_{i,j}).
				 \]
				 The conditions on the right hand side guarantee that every term in the meet forming $\hat{x}$ is greater than or equal to one of the terms whose meet is $x$, yielding the right-to-left implication. And if any of the conditions fails, then by the definitions of the sets $A_i$ and $A_{i, j, k}$, there is some element $n \in x$ such that $n \notin \hat{x}$, so $x\nleq \hat{x}$. Next, note that 
				\[
					x<\hat{x}\quad\Longleftrightarrow\quad x\leq \hat{x}~\&~x\neq \hat{x}\quad\Longleftrightarrow 
				\]
				\[ 
					(\ast)\text{ and }\hat{F}\neq F\text{ or }(\hat{F}=F\text{ and }G \setminus \hat{G}\neq \emptyset)\text{ or }(\exists i\in \hat{G})(\exists j\in \hat{G}_{i})(\hat{k}_{i,j}>k_{i,j}).
				\] 
				 Thus the relations $\le$ and $<$ on $L$ are computable from the normal forms of the elements of $L$. Meets can also be computed from the normal forms. Let 

				\begin{align*}
					G' & = G \cup \hat{G}\\
					F' & = (F \cup \hat{F}) \setminus G'\\
					G'_i & = \begin{cases}
								G_i &\text{ for $i \in G \setminus \hat{G}$} \\ 
								\hat{G}_i &\text{ for $i \in \hat{G} \setminus G$} \\ 
								G_i \cup \hat{G}_i &\text{ for $i \in G \cap \hat{G}$}
								\end{cases}\\
					k'_{i, j} & = \begin{cases}
								k_{i, j} &\text{ if $\hat{k}_{i, j}$ is undefined} \\ 
								\hat{k}_{i, j} &\text{ if $k_{i, j}$ is undefined} \\ 
								\min(k_{i, j}, \hat{k}_{i, j}) &\text{ if both are defined.}
								\end{cases}\\
				\end{align*}
				Then
				 \[x \wedge \hat{x} = \bigwedge_{i \in F'} A_i \wedge \bigwedge_{i \in G'} \bigwedge_{j \in G'_i} A_{i, j, k'_{i, j}} .\]

				 Joins can be computed by converting to a meet of joins of elements of $\Lambda$ using distributivity, then applying the rules for joins of elements of $\Lambda$, and finally reducing the meet to normal form.
				 \end{proof}
				 \begin{claim}
				 	$L$ is complete.
				 \end{claim}
	
				\begin{proof} 
				First, we consider an arbitrary (infinite) meet $\bigwedge_n x_{n}$. We may assume that $x_{n+1}\leq x_{n}$ and if the sequence is not eventually constant (and so its meet a finite one) that $x_{n+1}<x_{n}$. We claim any such meet is $0$. Suppose $x_{n}$ is represented by $\left\langle F^{n},G^{n},\left\langle G_{i}^{n}\mid i\in G^{n}\right\rangle\right\rangle $.
				If the $F^{n}$ are not eventually constant then there is an infinite set of $i$ that eventually appear in them and so the meet is below $A_{i}$ for infinitely many $i$.
				The only such element is $0$. Next, say the $F^{n}$ are eventually equal to $F$.
				If after they have settled down the $G^{n}$ are not eventually constant, and say equal to $G$, then there are infinitely many $i$ eventually appearing in the $G^{n}$
				and so the meet is below some $A_{i,j,k}$ for infinitely many $i$ and so also below infinitely many $A_{i}$. Therefore, it is once again $0$.
				So suppose $F^{n}$ and $G^{n}$ have stabilized by $n_{0}$.
				The only way $x_{n+1}<x_{n}$ for $n>n_{0}$ is for some $k_{i,j}^{n+1}$ to be smaller than $k_{i,j}^{n}$ for some $i\in G$.
				But this can happen only finitely often and so the meet eventually stabilizes, which is a contradiction.

				Next, consider an infinite join $\bigvee _{n}x_{n}$. Let
				\[
					y = \bigwedge \{z\mid \forall n(z\geq x_{n})\}
				\]
				which exists by the argument above. Clearly, $y$ is the least element of $L$ above every $x_{n}$.
				\end{proof}

				\begin{claim}
					$L$ is algebraic.
				\end{claim}
				\begin{proof}
				We determine the compact elements of $L$. It is easy to see that $1$ and $0$ are among them.
				If $P(i)$ fails let $j_i$ denote the unique witness such that $W_{h(i, j_i)}$ is infinite. Suppose $x\neq 0,1$ has the normal form
				\[\bigwedge _{i\in F}A_{i}\wedge \bigwedge _{i\in G}\bigwedge _{j\in G_{i}}A_{i,j,k_{i,j}}.\]
				We claim that $x$ is compact if and only if \[\tag{$\dagger$} (\forall i\in F)(P(i)) \text{ and }(\forall i \in G)(P(i) \text{ or } j_{i}\in G_{i}).\] First, suppose that $x$ is compact.
				If there is an $i'\in F$ such that $P(i')$ fails, then let $y_{k}$ be obtained by replacing the term $A_{i'}$ by $A_{i',j_{i'},k}$ in $x$, i.e., \[y_k = A_{i', j_{i'}, k} \wedge \bigwedge _{i\in F, i \neq i'}A_{i}\wedge \bigwedge _{i\in G}\bigwedge _{j\in G_{i}}A_{i,j,k_{i,j}}.\]
				It is clear that each $y_{k} < x$ and so $\bigvee_k y_{k}\leq x$.
				On the other hand, if $z$ is such that  $y_k \le z<x$, then, by our characterizations of the relations $\le$ and $<$, it must be some $y_{k'}$ for $k' \ge k$. It follows that $\bigvee_k y_{k}=x$, but no finite join suffices.

				Next, suppose that $P(i')$ fails for some $i'\in G$ and $j_{i'}\notin G_{i'}$. Let $y_{k}=x\wedge A_{i',j_{i'},k}$. An argument similar to the one above shows that $\bigvee_k y_{k}=x$ while no finite join is $x$.

				Next, we argue that if the condition $(\dagger)$ holds then $x$ is compact.
				Consider any $\bigvee_n x_{n}\geq x$. We may assume that if the join is not achieved at any finite stage then the $x_{n}$ are strictly increasing. Suppose \[x_n = \bigwedge _{i\in F^n}A_{i}\wedge \bigwedge _{i\in G^n}\bigwedge _{j\in G^n_{i}}A_{i,j,k^n_{i,j}}.\]
				It is clear from the characterization of $<$ that the $F^{n}$, $G^{n}$ and $G_{i}^{n}$ must eventually stabilize, say to $\bar{F}$, $\bar{G}$ and $\bar{G}_{i}$ for $i\in \bar{G}$.
				After stabilization, for $i\in \bar{G}$ such that $P(i)$ holds, or such that $P(i)$ fails but $j\in \bar{G}_{i}$ is not equal to $j_i$,  the $k_{i,j}^{n}$ are also eventually constant (since in either case, there are only finitely many of them). However, there must be some $i\in \bar{G}$ such that $P(i)$ fails and $j_{i}\in \bar{G}_{i}$, and for at least one such $i$, 	the $k_{i,j_{i}}^{n}$ must be unbounded. Let 
				\[
					H=\{i\mid (\forall m)(\exists n>m)(k_{i,j_{i}}^{n}>m)\}\quad\text{and}\quad
					K=H\cap \{i \mid \bar{G}_{i}=\{j_{i}\}\}.
				\]

				It is not difficult to see that $\bigvee_n x_n$ is represented by $\left\langle \bar{F}\cup K,\bar{G}\setminus K,\left\langle \hat{G}_{i}\mid i\in \bar{G}\setminus K\right\rangle \right\rangle $, where $
				\hat{G}_{i}=\bar{G}_{i}\setminus \{j_{i}\}$ if $i\in H$ and $\hat{G}_{i}=\bar{G}_{i}$ if $i\notin H$. Now,
				\begin{itemize}
					\item $\bar{F} \setminus F \subseteq G$, since $\bigvee_n x_n \ge x$, and so $(\bar{F} \cup K) \setminus F \subseteq G$
					\item $\bar{G} \subseteq G$, since otherwise, there is an $i \in K$ such that $i \notin G$, which means that $i \in F$, contradicting $(\dagger)$
					\item $(\forall i \in \bar{G} \setminus K)(\bar{G_i} \subseteq G_i)$, since if $P(i)$ holds, then $i \notin H$, and so $\hat{G}_i = \bar{G_i} \subseteq G_i$, and if $j_i \in G_i$, then  $\bar{G_i} \subseteq \hat{G}_i \cup \{j_i\} \subseteq G_i$
					\item $(\forall i \in K)(\bar{G_i} \subseteq G_i)$, since for all $i \in K$, $P(i)$ fails and therefore, by $(\dagger)$,  $j_i \in G_i$, and $\bar{G}_i = \{j_i\}$.
				\end{itemize}
				Therefore, for sufficiently large $n$, $x_n \ge x$.

				The above analysis shows that if $x$ is not compact it is the join of the compact elements below it:
				Define $y_{n}$ by replacing in the meet producing $x$ each $A_{i}$ ($i\in F$) such that $P(i)$ fails by $A_{i,j_{i},n}$ and,
				for each $i\in G$ for which $P(i)$ fails and $j_{i} \notin G_{i}$, adding $A_{i,j_{i},n}$ to the meet.
				Our characterization of the compact elements shows that each $y_{n}$ is compact. Our analysis of the order shows that $\bigvee y_{n}=x$.
				\end{proof}
				This completes the proof of the theorem.
			\end{proof}
			\begin{corollary}[$\RCA_0$]
				The following principle is equivalent to $\ACA_0$:
				``For each countable distributive lattice $L$, there is a set consisting exactly of the compact elements of $L$.''
			\end{corollary}
			\begin{proof}[Proof sketch.]
				To prove $\ACA_0$ from this principle, use the following construction.
				Let $L_n$ have a top element $t_n$ preceded by a finite sequence if $n\in 0'$ and an $\omega$-sequence if $n\not\in 0'$.
				Let $L$ be the sum of the linear orders, so that $t_n$ is compact iff $n\in 0'$.
				Then $L$ is a linear order, and hence in particular a distributive lattice.
			 \end{proof}
			\begin{figure}
				\[
					 \xymatrix{
					 						& \Pi^{1}_{1}-CA_{0}\ar[d]\ar[dr]	&				\\
											& ACA_{0}\ar[d]\ar[dr]\ar[dl]		& GS\ar[d]		\\
						RT^{2}_{2}\ar[dr]	& WKL_{0}\ar[d]						& GSD\ar[dl]	\\
											& RCA_{0}							&				\\
					 }
				\]
				\caption{Reverse mathematics of the Gr\"atzer-Schmidt theorem over $\RCA_{0}$.}\label{rm}
			\end{figure}
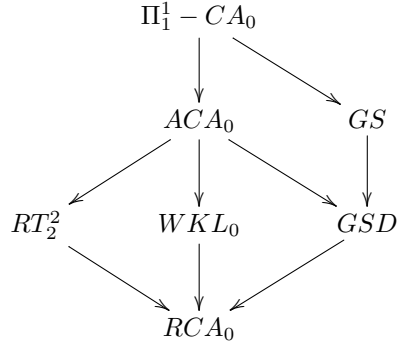
			\begin{corollary}\label{enabled}
				$\WKL_{0}+\RT_{2}^{2}\not\models\GSD$. 
			\end{corollary}
			\begin{proof}
				As the set of compact elements of a computable congruence lattice is $\Sigma _{2}^{0}$, the construction for Theorem \ref{d3hard} guarantees that
				any standard model of $\GSD$ includes a set $C$ such that the complete $\Pi _{3}^{0}$ set is $\Sigma _{2}^{0}$ in $C$ and so $C^{\prime\prime}\ge _{T}0^{\prime\prime \prime }$.
				There are, however, standard models of $\WKL_{0}+\RT_{2}^{2}$ in which all sets are low$_{2}$ \cite{CJS}, so $C''\equiv_{T}0''$.
			\end{proof}

			\begin{rem}
			\label{PudCon}
			Let $L$ be a countable algebraic lattice and let $K$ be the set of its compact elements, which is an upper semilattice. Pudl\'ak's proof \cite{MR0429699} of the Gr\"atzer-Schmidt Theorem proceeds by constructing a ``$K$-valued graph'' $(A, r, h)$, where $A$ is a set of vertices, $r$ a set of (undirected) edges, and $h: r \rightarrow K$ a surjective ``coloring'' of each edge by a compact element. A mapping $f: A \rightarrow A$ is said to be \emph{stable} if it respects the coloring in the following sense: for every edge $\{a, b\} \in r$, either $f(a) = f(b)$ or $h(\{a, b\}) = h(\{f(a), f(b)\})$. Then letting $F$ be the family of all stable mappings on $A$, the unary algebra $(A, F)$ satisfies the requirements of the theorem, i.e., $\text{Con}(A, F)$ is isomorphic to $L$.

			An inspection of this construction reveals that the $K$-valued graph $(A, r, h)$ is computable in $K$. Further, it suffices to choose a countable subfamily $\{f_n \mid n \in \omega\}  \subseteq F$ of stable mappings that are uniformly computable in $K$, so that $\text{Con}(A, \{f_n \mid  n \in \omega\}) \cong L$. 

			For $a, b \in A$, let $a \sim_x b$ if there is a path in $(A, r, h)$ connecting $a$ and $b$ all of whose edges are colored with compact elements that are less than or equal to $x$. It can then be shown that the map $\varphi : x \mapsto\, \sim_x$ is an isomorphism between $L$ and $\text{Con}(A, \{f_n \mid n \in \omega\})$. Moreover, $\varphi$ is $\Sigma^0_1$-definable in $K$. In particular, there is a presentation of $\text{Con}(A, \{f_n \mid n \in \omega\})$ that is arithmetical in $K$.

 			\end{rem}
				
			\begin{proposition}
				We have the following provability results:
				\begin{enumerate}
					\item $\PCA_0\vdash \GS$.
					\item $\ACA_{0}\vdash\GSD$.
				\end{enumerate}
			\end{proposition}
			\begin{proof}[Proof]
				For (1), note that $\PCA_0$ guarantees the existence of the set $K$ of compact elements in a given lattice $L$, and by the remark above, the congruence lattice and the isomorphism can be chosen to be arithmetical in $K$.

				For (2), Theorem \ref{distr3} shows that the set of compact elements of a computable algebraic distributive lattice is $\Pi _{3}^{0}$, and thus the congruence lattice and the isomorphism are, in this case, arithmetical.
			\end{proof}

		\subsection{Modular lattices}
			While we do not know whether the set of compact elements in a modular lattice must be $\Pi^0_3$, we do know that the characterization of compact elements in Theorem \ref{distr3} does not extend from distributive to modular lattices.

			The following fact is well-known:
			\begin{lem}\label{lem:compactly_generated}
				In an algebraic lattice, each element is the supremum of the compact elements below it.
			\end{lem}
			\begin{rem}
				An example to keep in mind: consider $(\omega+1)\times 2$.
				Let $a=(\omega,1)$. Then $a$ is not compact.
			\end{rem}

			\begin{thm}
				Let $L$ be a modular algebraic lattice and $a$ be in $L$. If $a$ is compact, then for each interval $(b, a)$, there is a $c \in (b, a)$ such that the interval $(c, a)$ is empty (we say that \emph{$a$ covers $c$}). However, the converse does not hold.
			\end{thm}
			\begin{proof}
				If $a$ does not cover any element in $(b,a)$, one can construct an infinite chain whose supremum is $a$ but for any finite subchain, the supremum is strictly below $a$, contradicting compactness.

				For a counterexample to the converse, consider the countably-infinite dimensional vector space $V$ over $\mathbb Z/2\mathbb Z$ consisting of all finite subsets of $\mathbb N$, viewed as finite characteristic functions, with mod-two addition or equivalently:
				\[
					A + B = (A \setminus B) \cup (B\setminus A).
				\]
				Let $V_n$ be the subspace of $V$ consisting of subsets of $\{1,\dots,n\}$. Then the supremum of $\{V_n:n\in\omega\}$ is $V$ but clearly the supremum of any finite subset of the $V_n$ is contained in some $V_k$.  On the other hand each proper subspace of $V$ is contained in a codimension 1 subspace.
			\end{proof}

	\section*{Acknowledgments}
		This paper is an extended and corrected version of the conference paper \cite{MR2545880}.
		The authors thank Bakh Khoussainov for useful suggestions.
		This work was partially supported by a grant from the Simons Foundation (\#315188 to Bj\o rn Kjos-Hanssen).
		Paul K. L. V.\ Nguyen was partially supported by the National Science Foundation under Grant No. 0841223.
		Richard A.\ Shore was partially supported by NSF grants DMS-0852811 and DMS-1161175.
		Bj\o rn Kjos-Hanssen was partially supported by NSF grant DMS-0901020.
	\bibliography{Graetzer-Schmidt}
\end{document}